\documentclass[12pt,a4paper]{article}
\usepackage[T1]{fontenc}
\usepackage[utf8]{inputenc}
\usepackage[british]{babel}
\usepackage{amsmath,amssymb,amsthm}
\usepackage{tensor}
\usepackage{xcolor}
\usepackage[all]{xy}
\usepackage{cite}
\usepackage{hyperref}

\newcommand{\p}[1]{\left(#1\right)}

\let\conj=\induced

\def\x{\times}

\def\Soc{\textup{Soc}}
\def\Ker{\textup{Ker}}

\def\Soc{\textup{Soc}}
\def\Fix{\textup{Fix}}

\def\Z{\textup{Z}}

\newtheorem{teo}{Theorem}[section]
\newtheorem{prop}[teo]{Proposition}
\newtheorem{lemma}[teo]{Lemma}
\newtheorem{cor}[teo]{Corollary}
\newtheorem{athm}{Theorem}

\theoremstyle{definition}
\newtheorem{defi}[teo]{Definition}

\theoremstyle{remark}

\theoremstyle{definition}
\newtheorem{ex}[teo]{Example}

\title{On products of abelian skew braces}
\author{A. Ballester-Bolinches%
\thanks{Departament de Matem\`atiques, Universitat de Val\`encia, Dr.\ Moliner, 50, 46100 Burjassot, Val\`encia, Spain; \texttt{Adolfo.Ballester@uv.es}, \texttt{Ramon.Esteban@uv.es}, \texttt{Pedro.A.Perez@uv.es}; ORCID 0000-0002-2051-9075, 0000-0002-2321-8139,  0009-0009-7082-9002}
\and R. Esteban-Romero\addtocounter{footnote}{-1}\footnotemark \and P. P\'erez-Altarriba\addtocounter{footnote}{-1}\footnotemark}
\begin{document}
\date{}
\maketitle
\begin{abstract}
  The main objective of this paper is to study factorisations of skew left braces through abelian subbraces. We prove a skew brace theoretical analog of the classical Itô's theorem about product of two abelian groups: if $B = A_1A_2$ is a skew brace which is the product of two abelian skew subbraces $A_1$ and $A_2$, and $A_1$ is a left and right ideal of $B$, then the commutator ideal $[B, B]^B$ of $B$ is an abelian brace. If $A_1$ is a left (non-necessarily right) ideal of $B$, we show that there exists a strong left ideal of $B$ contained in $A_1$ or $A_2$. We also show factorisations of relevant ideals of factorised braces that are sums and products of abelian subbraces. 
  
  \emph{Keywords: factorised skew brace, abelian skw brace, trifactorised group}

  \emph{Mathematics Subject Classification (2020):
    16T25,  Yang-Baxter equations 
    81R50, 
    20C99, 
    20D40,  Products of subgroups of abstract finite groups}
\end{abstract}

\section{Introduction}

The Yang-Baxter equation is an important equation in pure mathematics and physics. One of the fundamental challenges related with this equation is the search of solutions. An important family of solutions can be constructed by means of linear extensions of maps $r\colon X\times X \rightarrow X\times X$, where $X$ is a basis of a vector space. They are called \emph{combinatorial} solutions. To construct and understand combinatorial solutions with some prescribed properties, left skew braces introduced in \cite{GuarnieriVendramin17, Rump07} have emerged as a remarkable algebraic structure. 

A \emph{left skew brace} $(B, {+}, {\cdot})$ consists of a set with two binary operations $+$ and~$\cdot$ such that $(B, {+})$ and $(B, {\cdot})$ are groups linked by a distributive-like law $a(b+c)=ab-a+ac$ for $a$, $b$, $c\in B$. When the operations on $B$ are clear from the context, we will refer to a left skew brace as $B$ instead of $(B, {+}, {\cdot})$. For simplicity, we will use \emph{skew brace} with the meaning of \emph{left skew brace}.

Factorised skew braces, or skew braces that are a sum or a product of some of their left (right) ideals are important for a better understanding of multipermutational combinatorial solutions of the Yang-Baxter equation. In this context, skew brace theoretical analogs of some important results on factorised groups play a significant role. 

The aim of this paper is to show a skew brace theoretical analog of Itô’s theorem on metabelian groups, and arises after a careful study of the results on factorised skew braces showed in \cite{JespersKubatVanAntwerpenVendramin19, Tsang_2024}.

In order to state our main result, we need to introduce some terminology. Let $(B,+,\cdot)$ be a skew brace. A \emph{subbrace} of $B$ is a subgroup of $(B, {+})$ which is also a subgroup of $(B, {\cdot})$. Both operations in $B$ can be related by the so-called {\it star product}: $a\ast b = -a + ab -b$, for all $a,b\in B$. Indeed, both group operations coincide if and only if, $a\ast b=0$ for all $a,b\in B$; in this case, $B$ is said to be a \emph{trivial skew brace}. Groups are precisely the trivial skew braces. We say that $B$ is \emph{abelian} if $B$ is a trivial skew brace and $(B, {+})$ is abelian, that is, $B$ is just an abelian group. 

The \emph{opposite} skew brace $B^{op}$ of $B$ is just the set $B$ with the same multiplicative group $(B, {\cdot})$ and the additive group $(B, +^{op})$ given by $a +^{op} b = b +a$ for all $a, b \in B$. 

Given two subsets $X,Y \subseteq B$, we write $\langle X \rangle_+,  [X,Y]_+$ for the subgroup generated by $X$ and the commutator of $X$ and $Y$ in $(B,+)$, respectively, and we write $\langle X\rangle_\bullet, [X,Y]_\bullet$ for the subgroup generated by $X$ and the commutator of $X$ and $Y$ in $(B,\cdot)$, respectively. The subgroup of $(B, +)$ generated by the set $\{x * y : x \in X, y \in Y\}$ is denoted by $X * Y$. 

A subbrace $L$ of $B$ is said to be a \emph{left ideal} (respectively \emph{right ideal}) of $B$ if $B \ast L \subseteq L$ (respectively $L \ast B \subseteq L$). A \emph{strong left ideal} $I$ of $B$ is a left ideal such that $(I,+)$ is a normal subgroup of $(B,+)$.  An \emph{ideal} of $B$ is a strong left ideal which is a right ideal of $B$. Note that every subgroup of a trivial skew brace is a left and right ideal, and $B*B$ is an ideal of $B$. 

The \emph{commutator} $[B,B]^B$ of $B$ is defined as the ideal of $B$ generated by the set
$\{[B,B]_+ \cup [B,B]_\bullet \cup \{ab - (a+b):\, a, b \in B\}$. This important ideal of $B$ was introduced in \cite{BournFacchiniPompili23} and plays a key role in the study of nilpotency and solubility of skew  braces (see \cite{BallesterEstebanFerraraPerezCTrombetti-arXiv-cent-nilp, BallesterEstebanJimenezPerezC24-solubleskewbraces}). The results of these papers show that the skew brace commutator is the true analog of the group commutator. In fact, $[B,B]^B$ is the smallest ideal of $B$ with quotient abelian, and $[B,B]^B = B*B + [B, B]_+$ (see \cite[Theorem ~3.3]{BallesterEstebanFerraraPerezCTrombetti-arXiv-cent-nilp}). Note that $B*B$ is a proper ideal of $[B,B]^B$ in general: if $B$ is a non-abelian trivial brace, then $0 = B*B \neq [B,B]^B = [B,B]_+$. 

The following analog of Itô's theorem for skew braces was proved in \cite{Tsang_2024}.

\begin{teo}[{\cite[Theorem 1.5]{Tsang_2024}}]\label{teo-ito-tsang}
Let $B$ be a skew brace and let be $A_1,A_2$ subbraces of $B$ satisfying the following conditions:
\begin{enumerate}
\item $B=A_1+A_2$ or $B=A_1A_2$.
\item $A_1$ and $A_2$ are trivial skew braces.
\item $A_1$ and $A_2$ are left and right ideals of $B^{op}$.
\end{enumerate}
Then $B*B$ is a trivial skew brace.
\end{teo}

Itô’s theorem may be recovered from Theorem~\ref{teo-ito-tsang}, when $B^{op}$ is a trivial skew  brace.  Note that Theorem \ref{teo-ito-tsang} depends on some properties of the skew brace and the opposite skew brace. For practical and computation purposes it would be much better to impose conditions only on the skew brace, otherwise be would need to compute the opposite skew brace first.

The first main result of this paper is a skew brace theoretical analog of Itô's theorem under conditions that depend only on the skew brace. Furthermore, no extra conditions on the second factor are needed.

\begin{athm}\label{teo-ito-general}
Let $B$ be a skew brace and let $A_1,A_2$ be subbraces of $B$ satisfying the following conditions:
\begin{enumerate}
\item $A_1$ and $A_2$ are abelian skew braces.\label{teo-ito-general-cond-1}
\item $B=A_1+A_2$ or $B=A_1A_2$.\label{teo-ito-general-cond-2}
\item $A_1$ is a left and right ideal of $B$.\label{teo-ito-general-cond-3}
\end{enumerate}
Then $[B,B]^B$ is an abelian skew brace.
\end{athm}

If $A_1$ is a strong left ideal and $A_2$ is a right ideal of $B$, Theorem~\ref{teo-ito-general} does not hold as the following example shows. Hence Condition~\ref{teo-ito-general-cond-3} is, in some sense, a minimal one. 

\begin{ex}
The skew left brace $B=S(24,811)$ has additive group $C_2\x C_2\x S_3$ and multiplicative group $C_2\x A_4$. $B$ contains a strong left ideal $A_1$ and a right ideal $A_2$, where $A_1$ and $A_2$ are abelian braces associated to the groups $C_3$ and $C_2\x C_2\x C_2$. It is clear that $B=A_1+A_2=A_1A_2$. We have that $B*B$ is not a trivial brace, and hence $[B,B]^B$ is not an abelian brace. 
\end{ex}

As in the theory of factorised groups, in the structural study of factorised skew braces seems interesting to know which relevant ideals inherit the factorisation. In the second part of the paper, we shall be concerned with this property, which we formalise with suitable definition. It is inspired by the concept of factorised subgroup of a group which is the product of two subgroups introduced by Wielandt 
(see\cite[Lemma~1.1.1]{AmbergFranciosiDeGiovanni92}).

\begin{defi}
A subbrace $S$ of a skew brace $B = A + C$ which is the sum of two subbraces $A$ and $C$ is called \emph{factorised} if $S = (A \cap  S) + (C \cap S)$ and $A \cap C$ is contained in $S$.
\end{defi}

Note that a weaker notion of factorised subbrace is presented in \cite[Definition~2.17]{JespersKubatVanAntwerpenVendramin19}.

Our next result shows that some important left ideals of skew braces which are the sum and product of trivial and abelian subbraces are factorised (see \cite{JespersKubatVanAntwerpenVendramin19} for related results). 

For a skew brace $B$, we can consider the set of fixed elements of $B$ by the action $\lambda$,
\[ \Fix(B)  = \{a \in B\,|\, \lambda_b(a) = a, \  \text{for every $b \in B$}\},\]
which turns out to be a left ideal, the \emph{socle} of $B$,
\begin{align*}
\Soc(B)  & = \{a \in B\, |\, \lambda_a(b) = b, \ a+b = b+a, \ \text{for every $b\in B$}\} \\
&  = \Ker \lambda \cap \Z(B,+),
\end{align*}
and the \emph{centre} of $B$:  $Z(B) = \Soc(B) \cap Z(B, \cdot)$

It is  well-known that $\Soc(B)$ and $Z(B)$ are ideals of $B$.

\begin{athm}\label{teo2}
Let the non-zero skew brace $B=A_1+A_2 = A_1A_2$ be the sum of two subbraces $A_1$ and $A_2$. 
\begin{enumerate}
\item If $A_1$ and $A_2$ are trivial skew braces, then $\Fix(B)$ and $\ker\lambda$ are factorised subbraces of $B$.

\item If $A_1$ and $A_2$ are abelian skew braces, then $\Soc(B)$ and $Z(B)$ are factorised ideals of $B$. 
\end{enumerate}
\end{athm}

The centre of a group which is the product of two subgroups is not factorised in general (see\cite[Example~7.1.5]{AmbergFranciosiDeGiovanni92}). Hence the hypothesis on $A_1$ and $A_2$ in Statement~2 of Theorem~\ref{teo2} is necessary.

Strong left ideals provide a framework in which the results of factorisations of skew braces can be studied (see \cite{JespersKubatVanAntwerpenVendramin19}) and are closely related with decomposability of solutions of the Yang-Baxter equation.  Our third main result shows the existence of strong left ideals of a factorised skew brace $B$ which is the sum of two abelian subbraces. 

\begin{athm}\label{teo3}
Let the finite non-zero skew brace $B=A_1+A_2$ be the sum of two abelian skew subbraces $A_1$ and $A_2$. If $A_1$ is a left ideal of $B$, there exists a non-zero strong left ideal of $B$ contained in $A_1$ or $A_2$.
\end{athm}

Our approach is based on the intimate relationship between skew braces and trifactorised groups showed in \cite{BallesterEsteban22,ballesterbolinches2025categoriesskewleftbraces}. Results of these papers show that this is a natural setting in which skew brace factorisations can be studied (see \cite[Section~6]{BallesterEsteban22}).

\section{Preliminaries}\label{sec-traduct}

Given a skew brace $(B,+,\cdot)$, the action of the multiplicative group $C=(B, \cdot)$ on the additive group $K=(B, +)$, called the \emph{lambda action}, is defined by $\lambda(a)=\lambda_a$ for every $a\in B$, where $\lambda_a(b)=-a+ab$ for all $a, b\in B$. The identity map $\delta\colon C\longrightarrow K$ is a bijective derivation with respect to the lambda action. Let $G=[K]C$ be the semidirect product of $K$ and $C$ with respect to the lambda action.

As usual, we identify $K$ with the normal subgroup $\{(k, 1)\;|\; k\in K\}$ and $C$ with the subgroup $\{(0, c)\;|\; c\in C\}$. Then $G$ can be viewed as a split extension $G = KC$ and $K \cap C = 1$, and where the lambda action of $C$ on $K$ is just the conjugation action within $G$. Furthermore, $G$ possesses a triple factorisation $G=KC=KD=CD$, where $D=\{\delta(c)c\;|\; c\in C\}$ is a subgroup of $G$, and $K\cap C=K\cap D=C\cap D=1$. $(G,K,D,C)$ is called the \emph{large trifactorised group} associated to $B$.  The passage from $C$ to $K$ and vice versa will be done by means of the derivation $\delta\colon C\longrightarrow K$.

 We compute the conjugation $\conj{b}a = bab^{-1}$ and the commutator $[a, b] = aba^{-1}b^{-1}$, for $a, b \in G$, in the large trifactorised group $G$ associated to the skew brace $B$.

The following presents some basic facts about the relationships between the subbrace and ideal structure of a skew brace and the subgroup structure of its associated large trifactorised group.

\begin{prop}[{\cite[Proposition 5.2(1)]{ballesterbolinches2025categoriesskewleftbraces}}]
A subset $L$ of $K$ is a subbrace of $B$ if and only if $LD\cap LC$ is a subgroup of $G$.
\end{prop}

The large trifactorised group associated with a subbrace of $B$ is given in the following.

\begin{prop}[{\cite[Proposition 5.4]{ballesterbolinches2025categoriesskewleftbraces}}]\label{prop-3factAss-SubBr}
Let $L\subseteq K$ be a subbrace of $B$. Then $T=L\sigma^{-1}(L) = LD\cap LC$ is a subgroup of $G$ and 
$$(LD\cap LC,L, LC\cap D,LD\cap C)$$
is the large trifactorised group associated with $L$.
\end{prop} 

Let $A$, $A_1$ and $A_2$ be subbraces of $B$. We identify them as the subgroups $L$, $L_1$ and $L_2$ of $K$ in $G$ respectively. If $a \in A_1$ and $b \in A_2$, then $a*b$ corresponds to the element $[\delta^{-1}(a), b]$ of $G$ which belongs to $K$. Therefore the additive group of $A_1*A_2$ can be identified  as the subgroup $[\delta^{-1}(L_1),L_2]$ of $G$. Note that $[\delta^{-1}(L_1),L_2]$ is contained in $K$. By \cite[Lemma 5.3]{ballesterbolinches2025categoriesskewleftbraces}, $\delta^{-1}(L_1)=L_1D\cap C$. Therefore we have

\begin{prop}\label{prop-add-gr-deriv}
The subgroup $A_1*A_2$ of $K$ corresponds to the subgroup $[L_1D\cap C,L_2]$ of $G$. In particular, the additive group of $B*B$ corresponds to the normal subgroup $[C,K]$ of $G$.
\end{prop}

Next we show how to compute some important subbraces of $B$ in the large trifactorised group associated to $B$. 
\begin{prop}\label{prop-trad-prop-br}
\begin{enumerate}
\item $A$ is a trivial skew brace if and only if, $[LD\cap C,L]=1$.
\item $A$ is a left ideal if and only if, $[C,L]\leq L$.
\item $A$ is a right ideal if and only if, $[LD\cap C,K]\leq L$.
\item $A$ is a strong left ideal if and only if, $L$ is normal in $G$.
\end{enumerate}
\end{prop}
\begin{proof}
\begin{enumerate}
\item By Proposition~\ref{prop-3factAss-SubBr}, we have that $(T,L,T\cap D,T\cap C)$ with $T=LD\cap LC$ is the large trifactorised group associated to $L$. Therefore, $A$ is a trivial skew brace if and only if, $T\cap C$ acts trivially on $L$, that is, $[LD\cap C,L] = 1$.
\item By {\cite[Proposition 5.2.2]{ballesterbolinches2025categoriesskewleftbraces}}, $A$ is a left ideal of $B$ if and only if, $C$ normalises $L$, that is, $[C,L]\leq L$.

\item $A$ is a right ideal if and only if,  $A*B\subseteq A$. By Proposition \ref{prop-add-gr-deriv}, the latter statement is equivalent to $[LD\cap C,K]\leq L$.

\item  This is exactly {\cite[Proposition 5.2.3]{ballesterbolinches2025categoriesskewleftbraces}}.\qedhere
\end{enumerate}
\end{proof}

We close this section with two applications of the above propositions.

\begin{cor}\label{prop-ab-br-eq}
The following statements are equivalent:
\begin{enumerate}
\item $B$ is an abelian skew brace.\label{prop-ab-br-eq-cond-1}
\item $G$ is abelian.\label{prop-ab-br-eq-cond-2}
\item $K$ is abelian and $[C,K]=1$.\label{prop-ab-br-eq-cond-3}
\end{enumerate}
\end{cor}
\begin{proof}
\ref{prop-ab-br-eq-cond-1} implies \ref{prop-ab-br-eq-cond-2}. Since $B$ is an abelian skew brace, then  $K$ and $C$ are abelian and $[C,K]=1$ by Proposition \ref{prop-add-gr-deriv}. Hence $G=[K]C$ is abelian.
 
It is clear that \ref{prop-ab-br-eq-cond-2} implies \ref{prop-ab-br-eq-cond-3}. Now if $K$ is abelian and $[C,K]=1$, then $B$ is an abelian skew brace by Proposition \ref{prop-trad-prop-br}. Therefore \ref{prop-ab-br-eq-cond-3} implies \ref{prop-ab-br-eq-cond-1}, and the circle of implications is complete. 
\end{proof}

\begin{cor}\label{prop-add-mul-br-eq}
\begin{enumerate}
\item $B=A_1+A_2$ if and only if, $K=L_1L_2$.
\item $B=A_1A_2$ if and only if, $C=(L_1D\cap C)(L_2D\cap C)$.
\end{enumerate}
\end{cor}
\begin{proof}
\begin{enumerate}
Statement 1 is clear. Now, $B = A_1A_2$ if and only if, $C = \delta^{-1}(L_1)\delta^{-1}(L_2)$. By \cite[Lemma 5.3]{ballesterbolinches2025categoriesskewleftbraces}, $\delta^{-1}(L_i)=L_iD\cap C$, $i = 1,2$. Hence Statement 2 holds.\qedhere
\end{enumerate}
\end{proof}

\section{Proof of Theorem~\ref{teo-ito-general}}\label{sec-ito-general}

In the following, we prepare the way to prove Theorem~\ref{teo-ito-general}. We begin with a basic result.

\begin{lemma}\label{lemma-prod-sum-ideals-is-B}
Let $B$ be a skew brace, and let $A_1$ be a subbrace of $B$ and $A_2$ a left ideal of $B$, then $A_1+A_2= A_1A_2$.
\end{lemma}
\begin{proof}
Since $A_2$ is $\lambda$-invariant, it follows that $a_1+a_2=a_1\lambda_{a_1^{-1}}(a_2) \in A_1A_2$. Furthermore, $a_1*a_2 \in B*A_2 \subseteq A_2$. Therefore $a_1a_2=a_1+a_1*a_2+a_2 \in A_1 + A_2$ and so $A_1+A_2= A_1A_2$.\qedhere
\end{proof}
\begin{prop}\label{prop-br-comm-car}
Let $B$ be a skew brace, then 
$$[B,B]^B=B*B+[B,B]_+=B*B+[B,B]_\bullet.$$
\end{prop}
\begin{proof}
By Theorem \cite[Theorem ~3.3]{BallesterEstebanFerraraPerezCTrombetti-arXiv-cent-nilp} $[B,B]^B = B^2 + [B, B]_+$. For the other equality, notice that $B/B^2$ is a trivial skew brace, hence,
\begin{align*}
(B^2[B,B]_\bullet)/B^2=[B/B^2,B/B^2]_\bullet=[B/B^2,B/B^2]_+=(B^2+[B,B]_+)/B^2
\end{align*}
Therefore, $B^2[B,B]_\bullet=B^2+[B,B]_+$. The result follows from Lemma \ref{lemma-prod-sum-ideals-is-B} because $B^2$ is left ideal.
\end{proof}

Although the next results will be applied to large trifactorised groups, we consider convenient to prove them for trifactorised groups.   A trifactorised group is a $4$-tupla $(G,K,H,E)$, where $G$ is a group, $K$ is a normal subgroup of $G$, and $H$ and $E$ are subgroups of $G$ satisfying $G=KE=KH=HE, K\cap E=H\cap E=1$. These groups are naturally associated to skew braces (see \cite{ballesterbolinches2025categoriesskewleftbraces}). 

The next result appears in \cite[Propositions~2.1 and 2.2]{CascellaGariulo2025}. We give a slightly different proof for the sake of completeness.
\begin{prop}\label{prop-tsang-3.3}
Let $(G,K,H,E)$ be a trifactorised group, $L_1,L_2\leq K$ and $E_1,E_2\leq E$ such that $K=L_1L_2$, $E=E_1E_2$ and $[E_i,L_i]=1$ for $i=1,2$. Then $[E_1,L_2]$ and $[E_2,L_1]$ are normal in $G$ and $[E,K]=[E_1,L_2][E_2,L_1]$.
\end{prop}
\begin{proof}
Since $[E_1,K]=[E_1,L_1L_2]=[E_1,L_2]$ and $[E,L_2]=[E_1E_2,L_2]=[E_1,L_2]$, $K$ and $E$ normalises $[E_1,L_2]$ by \cite[Kapitel~III, Hilfssatz~1.6]{Huppert67}. Hence $[E_1,L_2]$ is normal in $KE=G$. Analogously, $[E_2,L_1]$ is normal in $G$.

Let $e\in E$ and $k\in K$, then there exists $l_i\in L_i$ for $i=1,2$ such that $k=l_1l_2$. Hence $[e,k]=[e,l_1l_2]=[e,l_1]\conj{l_1}{[e,l_2]} = [e,l_1]\conj{l_1}{[e_1,l_2]}$ for some $e_1 \in E_1$. Hence $[e, k] \in [E,L_1][E_1,L_2] = [E_2,L_1][E_1,L_2]$. Then $[E,K] \subseteq  [E_1,L_2][E_2,L_1]$ and $[E,K]=[E_1,L_2][E_2,L_1]$.
\end{proof}

Skew brace version of the above result provides a slightly stronger version of  {\cite[Proposition 3.3]{Tsang_2024}}.

\begin{cor}
Let $B$ be a skew brace and $A_1,A_2$ two trivial subbraces such that $B=A_1+A_2$ and $B=A_1A_2$, then $A_1*A_2$ and $A_2*A_1$ are strong left ideals of $B$. Furthermore, $B*B=A_1*A_2+A_2*A_1$.
\end{cor}

\begin{teo}\label{teo-3fact-ito-general}
Let $(G,K,H,E)$ be a trifactorised group, $L_1,L_2\leq K$ and $E_i=L_iH\cap E \leq E$  such that:
\begin{enumerate}

\item $E=E_1E_2$ and $K=L_1L_2$.\label{teo-3fact-ito-general-cond-2}
\item $[E,L_1]\leq L_1$ and $[E_1,K]\leq L_1$.\label{teo-3fact-ito-general-cond-3}
\item $L_iH\cap L_iE$ is an abelian subgroup of $G$ for $i=1,2$.\label{teo-3fact-ito-general-cond-1}
\item $K'H\cap E\leq E'([E,K]H\cap E)$.\label{teo-3fact-ito-general-cond-4}
\end{enumerate}
Then $K'[E,K]$ is abelian and $[K'[E,K]H\cap E,K'[E,K]]=1$.
\end{teo}
\begin{proof}
Let us denote by $T_i=L_iH\cap L_iE$, $i = 1,2$. Then $L_iE_i = L_i(L_iH\cap E) = T_i$, $i = 1,2$. By Statement \ref{teo-3fact-ito-general-cond-1}, $E_i$ is an abelian subgroup of $G$ and $[E_i,L_i] \leq [T_i, T_i] = 1$, $i = 1, 2$.  By Statement
 \ref{teo-3fact-ito-general-cond-3}, $E_2$ normalizes $L_1$. Therefore, $T_2T_1=L_2E_2L_1E_1=L_2L_1E_2E_1=KE=G$ and $G$ is the product of two abelian groups. By Itô's theorem (\cite[Kapitel~VI, Satz~4.4]{Huppert67}), $G'$ is abelian. In particular, $K'[E,K]$ is abelian.

By Proposition \ref{prop-tsang-3.3}, $[E,K]=[E_1,L_2][E_2,L_1]$.
By Statement \ref{teo-3fact-ito-general-cond-3}, $[E_1,L_2]$ and $[E_2,L_1]$ are contained in $L_1$. Therefore, $[E,K]\leq L_1$ and $[E,K]H\cap E\leq E_1$. Now, $(K'H\cap E)([E,K]H\cap E) = (K'H \cap E)[E, K]H \cap E = [E, K](K'H \cap E)H \cap E = [E, K](K'H \cap EH) \cap E = [E, K]K'H \cap E$. Hence, by Statement \ref{teo-3fact-ito-general-cond-4}, 
$[E, K]K'H \cap E \leq (E'([E,K]H\cap E))([E,K]H\cap E) \leq E'[E, K]H \cap E = E'([E, K]H \cap E) \leq E'E_1$. 

Moreover, since $K=L_1L_2$ and $L_1,L_2$ are abelian, it follows that $K'=[L_1,L_2]$. As $[[E_1, L_1], L_2] = 1$ ($[E_1, L_1] = 1$) and $[[L_2, E_1], L_1] =1$  ($[L_2, E_1] \leq L_1$), it follows that $[[L_1, L_2], E_1] = 1$ by the three subgroups lemma (\cite[Kapitel~III, Hilfssatz~1.10]{Huppert67}). Hence $[E_1, K'] = 1$. Since $G'$ is abelian, $[E'E_1,K'[E,K]]= [E_1, K'] = 1$ (note that $[[E,K],E_1]\leq [L_1,E_1]=1$). Consequently, $[K'[E,K]H\cap E,K'[E,K]]=1$, as required. \qedhere

\end{proof}

\begin{proof}[{Proof of Theorem~\ref{teo-ito-general}}]
Let $(G,K,D,C)$ be the large trifactorised group associated to $B$ and let $L_i$ be the additive group of $A_i$, $i = 1,2$. We will apply Theorem \ref{teo-3fact-ito-general} to this factorised group , where $D = H$ and $C = E$. According to Proposition~\ref{prop-3factAss-SubBr}, $E_i=L_iD\cap C$ is a subgroup of $G$, $i = 1,2$. Suppose that $B = A_1A_2$. Since $A_1$ is a left ideal of $B$, it follows that $B = A_1 + A_2$ by Lemma \ref{lemma-prod-sum-ideals-is-B}. Suppose $B = A_1+ A_2$. Since $A_1$ is a left ideal of $B$, it follows that $B = A_1A_2$ by Lemma \ref{lemma-prod-sum-ideals-is-B}. In any case, we have $B = A_1A_2 = A_1 + A_2$. By Corollary~\ref{prop-add-mul-br-eq}, $K = L_1L_2$ and $C = E_1E_2$. Since $A_i$ is abelian, $i = 1, 2$, it follows that $L_iH\cap L_iE$ is an abelian subgroup of $G$ for $i=1,2$ by Proposition~\ref{prop-3factAss-SubBr} and Corollary~\ref{prop-ab-br-eq}. Furthermore, since $A_1$ is left and right ideal of $B$, it follows that $[C, L_1] \leq L_1$ and $[E_1, K] \leq L_1$ by Proposition~\ref{prop-trad-prop-br}. By Proposition~\ref{prop-3factAss-SubBr}, the subbrace $[B,B]_+$ corresponds to the subgroup $K'D\cap C$ of $C$, $[B,B]_\bullet$ is just $C'$ in $C$ and the subbrace $B*B$ of $B$ corresponds to the subgroup $[C, K]D \cap C$ of $C$. By Proposition~\ref{prop-br-comm-car}, $[B,B]_+\leq B*B+[B,B]_\bullet=[B,B]_\bullet+B*B=[B,B]_\bullet(B*B)$. Therefore $K'D\cap C \leq C'([C, K]D \cap C)$. Applying Theorem~\ref{teo-3fact-ito-general}, it follows that $K'[C,K]$ is abelian and $[K'[C,K]D\cap C,K'[C,K]]=1$. Since $K'[C,K]$ corresponds to the skew brace commutator $[B,B]^B$ and $([K'[C,K]D\cap K'[C,K]C, K'[C,K], K'[C,K]C\cap D, K'[C,K]D\cap C)$ is the large trifactorised groups associated to $[B,B]^B$, we can apply Corollary~\ref{prop-ab-br-eq} to conclude that $[B,B]^B$ is an abelian skew brace. \qedhere
\end{proof}

\section{Proof of Theorem~\ref{teo2}}

\begin{proof}[{Proof of Theorem~\ref{teo2}}]
\begin{enumerate}
\item Let $a\in\Fix(B)$, then there exists $a_i\in A_i$, $i=1,2$, such that $a=a_1+a_2$. Let $b\in B$, then there exists $b_i\in A_i$, $i=1,2$, such that $b=b_1b_2$. Then $a = a_1+a_2= \lambda_{b_2}(a_1+a_2)=\lambda_{b_2}(a_1)+\lambda_{b_2}(a_2)=\lambda_{b_2}(a_1)+a_2$. Hence $\lambda_{b_2}(a_1)=a_1$. Since $A_1$ is trivial $\lambda_{b}(a_1)=\lambda_{b_1}\p{\lambda_{b_2}(a_1)}=\lambda_{b_1}(a_1)=a_1$. Hence $a_1\in\Fix(B)$. Analogously, $a_2\in\Fix(B)$. Therefore, $a\in (\Fix(B)\cap A_1)+(\Fix(B)\cap A_2)$. 
Furthermore, if $z \in A_1 \cap A_2$ and $b=b_1b_2$, $b_i\in A_i$, $i=1,2$, then $\lambda_{b}(z)=\lambda_{b_1}\p{\lambda_{b_2}(z)} = z$ since $A_i$, $i=1,2$, are trivial skew braces. Consequently, $\Fix(B)$ is a factorised left ideal of $B$. 

Let $a\in\ker\lambda$ and $b\in B$, then there exists $a_i,b_i\in A_i$ such that $a=a_1a_2$ and $b=b_1+b_2$. Hence $\lambda_{a_1}(b)=\lambda_{a_1}(b_1)+\lambda_{a_1}(b_2)=b_1+\lambda_{a_1}\p{\lambda_{a_2}(b_2)}=b_1+\lambda_a(b_2)=b_1+b_2=b$. Thus $a_1\in\ker\lambda$. Analogously $a_2\in\ker\lambda$. Therefore, $a\in(\ker\lambda\cap A_1)(\ker\lambda\cap A_2)$, and  $\ker\lambda =( \ker\lambda \cap A_1)+(\ker\lambda \cap A_2)$.
Furthermore, if $z \in A_1 \cap A_2$ and $b=b_1b_2$, $b_i\in A_i$, $i=1,2$, then $\lambda_z(b) = \lambda_{z}(b_1)+\lambda_{z}(b_2 ) = b_1 + b_2 = b$ since $A_1$ and $A_2$ are trivial skew braces. Therefore $\ker\lambda$ is factorised.

\item Assume that $A_1$ and $A_2$ are abelian skew braces. We prove first that $A_1\cap A_2\leq Z(B)$. Let $a\in A_1\cap A_2$ and $b\in B$. There exists $b_i,b_i'\in A_i$, $i=1,2$ such that $b=b_1+b_2=b_1'b_2'$. Therefore, $\lambda_a(b)=\lambda_a(b_1+b_2)=\lambda_a(b_1)+\lambda_a(b_2)=b_1+b_2$, $[a,b]_+=[a,b_1+b_2]_+=[a,b_1]_+=1$, $[a,b]_\bullet=[a,b_1'b_2']_\bullet=[a,b_1']_\bullet =1$. Hence $a\in Z(B)$. 

Let $a\in\Soc(B)=\ker\lambda\cap Z(B,+)$ and $b\in B$. By Statement~(1),  there exist $a_i\in \ker\lambda\cap A_i$, $i=1,2$, such that $a=a_1+a_2$. Let $b_i\in A_i$, $i=1,2$, such that $b=b_1+b_2$. Then, $1=[a,b_2]_+=[a_1+a_2,b]_+=[a_1 ,b_2]_+$ because $A_2$ is abelian. Since $A_1$ is abelian, it follows that $[a_1,b_1]_+=1$ and so $[a_1,b]_+=1$ and $a_1\in Z(B,+)$. Analogously $a_2\in  Z(B,+)$. Therefore, $a\in (\Soc(B)\cap A_1)+(\Soc(B)\cap A_2)$. Then $\Soc(B)=(\Soc(B)\cap A_1)+(\Soc(B)\cap A_2)$ and $\Soc(B)$ is a factorised ideal of $B$. .

Let $a\in Z(B)=\Soc(B)\cap Z(B,\cdot)$ and $b\in B$. Since $\Soc(B)$ is factorised and trivial, there exist $a_i\in \Soc(B)\cap A_i$, $i=1,2$, such that $a=a_1a_2$. Let  $b_i\in A_i$, $i=1,2$, such that $b=b_1b_2$. Then, $1=[a,b_2]_\bullet=[a_1a_2,b_2]_\bullet=[a_1,b_2]_\bullet$ because $A_2$ is abelian. Since $A_1$ is abelian $[a_1,b_1]_\bullet =1$. Hence $[a_1,b]_\bullet =1$ and $a_1\in Z(B,\cdot)$. Analogously $a_2\in Z(B,\cdot)$. Therefore, $a\in (Z(B)\cap A_1)+(Z(B)\cap A_2)$. Then $Z(B)=(Z(B)\cap A_1)+(Z(B)\cap A_2)$ and $Z(B)$ is a factorised ideal of $B$. 
\end{enumerate}
\end{proof}

\section{Proof of Theorem~\ref{teo3}}

We need a couple of preliminary results. 

\begin{lemma}\label{lemma-3fact-fact-centraliser}
Let $(G,K,H,E)$ be a trifactorised group and $L_1,L_2\leq K$ such that:
\begin{itemize}
\item $T_i=L_iH\cap L_iE$ is an abelian subgroup of $G$ for $i=1,2$.
\item $K=L_1L_2$ and $G=T_1T_2$.
\end{itemize}
Let $X$ be a subgroup of $G$ such that $X=(X\cap T_1)(X\cap T_2)$. Then $C_K(X)=(C_K(X)\cap L_1)(C_K(X)\cap L_2)$.
\end{lemma}
\begin{proof}
Let $k\in C_K(X)$. Let $k_i\in L_i$, $i=1,2$ such that $k=k_1k_2$. Note that $T_1$ centralises $k_1$ and if $t\in X\cap T_2$, we have that $1=[t,k]=[t,k_1k_2]=[t,k_1]\conj{k_1}{[t,k_2]}=[t,k_1]$. Hence $X\cap T_2$ centralises $k_1$ too. Therefore, $k_1$ centralises $(X\cap T_1)(X\cap T_2)=X$, and so $k_1\in C_K(X)$. Analogously, $k_2\in C_K(X)$. Therefore, $k\in (C_K(X)\cap L_1)(C_K(X)\cap L_2)$, and
 $C_K(X)=(C_K(X)\cap L_1)(C_K(X)\cap L_2)$.
\end{proof}

\begin{teo}\label{teo-3fact-sli-cont}
Let $(G,K,H,E)$ be a finite non-trivial trifactorised group, $L_1,L_2\leq K$ and $E_i=L_iH\cap E \leq E$ such that:
\begin{itemize}
\item $E=E_1E_2$ and $K=L_1L_2$.
\item $L_iH\cap L_iE$ is an abelian subgroup of $G$ for $i=1,2$.
\item $[E,L_1]\leq L_1$.
\end{itemize}
Then there exists a non-trivial normal subgroup of $G$ contained in $L_1$ or $L_2$.
\end{teo}
\begin{proof}
Arguing as in Proposition~\ref{teo-3fact-ito-general}, we have that $G=T_1T_2$ and $G'$ is abelian by Itô's theorem (\cite[Kapitel~VI, Satz~4.4]{Huppert67}). Since $G$ is non-trivial, $K$ is non-trivial, therefore, $L_1$ or $L_2$ is non-trivial. If $K'[E,K]=1$, every subgroup of $K$ is normal in $KE=G$. Hence $L_1$ and $L_2$ are normal subgroups of $G$, and the result is proved. Hence we may suppose that $K'[E,K] \neq 1$.

Consider $X=G'T_1\cap G'T_2=G'(T_1\cap G'T_2)=G'(T_2\cap G'T_1)=(T_1\cap G'T_2)(T_2\cap G'T_1)$. Since  $G'$, $T_1\cap G'T_2$ and $T_2\cap G'T_1$ are abelian, we can apply \cite[Corollary 1.3.5]{AmbergFranciosiDeGiovanni92}, to conclude that $X$ is nilpotent. Since $1 \neq K'[E,K] \leq K \cap X$, it follows that $K \cap Z(X) \neq 1$. By Lemma \ref{lemma-3fact-fact-centraliser}, $C_K(X)=(C_K(X)\cap L_1)(C_K(X)\cap L_2)$. Therefore, either $C_K(X)\cap L_1 \neq 1$ or $C_K(X)\cap L_2 \neq 1$. Suppose that $N=C_K(X)\cap L_1 \neq 1$. Note that $N\leq L_1\cap C_G(G')$.

Let us suppose by contradiction, that there are no non-trivial normal subgroups of $G$ contained in $L_1$ and $L_2$. Then by Lemma \ref{lemma-3fact-fact-centraliser}, $K\cap Z(G)=C_K(G)=(Z(G)\cap L_1)(Z(G)\cap L_2)=1$ because $Z(G)\cap L_i$ is a normal subgroup of $G$ contained in $L_i$, $i=1,2$. Write $C=C_G(T_1\cap C_G(G'))$. Note that $C$ contains $T_1G'$, therefore, $C$ is a normal subgroup of $G$. Since $T_1$, $T_2$ are abelian, $T_1\leq C$ and $G=T_1T_2$, it follows that $T_1Z(C) \leq C_G(T_1)$  and $T_2\cap T_1Z(C)\leq Z(G)$. Since $L_2\cap L_1Z(C)\leq K\cap (T_2\cap T_1Z(C))\leq K\cap Z(G)=1$, it follows
$$L_1(Z(C)\cap K)=L_1Z(C)\cap K=L_1Z(C)\cap L_1L_2=L_1(L_1Z(C)\cap L_2)=L_1.$$
Therefore, $Z(C)\cap K\leq L_1$. Since $Z(C)\cap K$ is a normal subgroup of $G$ contained in $L_1$, it follows that $Z(C)\cap K=1$. Therefore $1 \neq L_1\cap C_G(G')\leq Z(C)\cap K=1$, a contradiction.
\end{proof}

\begin{proof}[{Proof of Theorem~\ref{teo3}}]
Let $(G,K,D,C)$ be the large trifactorised group associated to $B$ and let $L_i$ be the additive group of $A_i$, $i = 1,2$. We will apply Theorem \ref{teo-3fact-sli-cont} to this factorised group , where $D = H$ and $C = E$. According to Proposition~\ref{prop-3factAss-SubBr}, $E_i=L_iD\cap C$ is a subgroup of $G$, $i = 1,2$.  By Corollary~\ref{prop-add-mul-br-eq}, $K = L_1L_2$ and $C = E_1E_2$. Since $A_i$ is abelian, $i = 1, 2$, it follows that $L_iH\cap L_iE$ is an abelian subgroup of $G$ for $i=1,2$ by Proposition~\ref{prop-3factAss-SubBr} and Corollary~\ref{prop-ab-br-eq}. Finally, since $A_1$ is left ideal of $B$, it follows that $[C, L_1] \leq L_1$ by Proposition~\ref{prop-trad-prop-br}. Applying Theorem \ref{teo-3fact-sli-cont}, there exists a non-trivial normal subgroup of $G$ contained in $L_1$ or $L_2$. By Proposition \ref{prop-trad-prop-br}, $N$ correspond to a non-zero strong left ideal of $B$ contained in $A_1$ or $A_2$.
\end{proof}

\section*{Acknowledgements}

The first and second authors are supported by the grant CIAICO/2023/007 from the Conselleria d'Educació, Universitats i Ocupació, Generalitat Valenciana.

\bibliographystyle{plain}
\bibliography{bibgroup}
\end{document}